\documentclass{article}
\usepackage{graphicx} % Required for inserting images
\usepackage[utf8]{inputenc}
\usepackage{t1enc}
\usepackage{amsmath}
\usepackage{arydshln}
\usepackage{amssymb}
\usepackage{amsthm}
\usepackage[hidelinks]{hyperref}
\usepackage{titling}

\title{On the number of $2$-dice games with prime power dice-size}
\author{Daniel Seress}
\predate{}
\postdate{}
\date{}

\begin{document}

\theoremstyle{definition}
\newtheorem{definition}{Definition}[subsection]
\newtheorem{remark}[definition]{Remark}
\newtheorem{proposition}[definition]{Proposition}
\newtheorem{example}[definition]{Example}
\newtheorem{lemma}[definition]{Lemma}
\newtheorem{theorem}[definition]{Theorem}
\newtheorem{corollary}[definition]{Corollary}
\newtheorem{conjecture}[definition]{Conjecture}
\newtheorem{problem}[definition]{Problem}

\maketitle

We have two dices which have $n$ labels. We want to assign a pair of labelings with positive integers to them such that the sums of the labels on them are the same and have the same distribution as at the standard labeling. We call such a pair of labelings a (dice) game of size $2$ (with dice size $m$).

Here we give a lower and an upper bound for the dice games of size $2$ with Sicherman dice of dice size $p^{k}$ where $p$ is a prime. We use the standard technique of generating polynomials to prove a non-closed combinatorial formula.
\begin{definition}
The \emph{size} of a dice is the number of its labels (counting repetitions). \cite{GR}

A dice of size $m$ is \emph{standard} if its labels are $1$ through $m$. \cite{GR}

A set of $n$ dice (each of size $m$) is called an \emph{$n$-dice game} (\emph{with dice-size $m$}). \cite{GR}

A game with $n$ standard dice (each of size $m$) is called a \emph{standard $n$-dice game} (\emph{with dice-size $m$}).
\end{definition}
The general problem is the following: Given $n$ and $m$, determine all possible $n$-dice games with dice-size $m$, so that the probability of obtaining any particular sum is the same as that obtained in the standard $n$-dice game with dice-size $m$.
\begin{definition}
Any labeling that appears on one of $n$ such dice is called
a \emph{solution of an $n$-dice game with dice-size $m$}. \cite{GR}

A game with the same distribution of sums as a standard game is called a \emph{sum-standard game}
\end{definition}
\cite{GR} counts \emph{dice} of a given size $m$ in $n$-dice games where $n$ is unspecified. Here we count \emph{$2$-dice games} with a given dice-size $m$.

Let $S(n, m)$ be the number of \emph{nonstandard} solutions of an $n$-dice game with dice-size $m$.

Let $(a_{1}, ..., a_{m})$ and $(b_{1}, ..., b_{m})$ be the labelings of two dice of size $m$. Let $f(x) = \sum_{i=1}^{m} x^{a_{i}}$ and $g(x) = \sum_{i=1}^{m} x^{b_{i}}$ be the polynomials assigned to them as written in \cite{GR}.

Then $f(1) = m$ and $g(1) = m$. The coefficients of $f$ and $g$ are nonnegative integers.

The polynomial assigned to the standard dice of size $m$ is $\sum_{i=1}^{m} x^{i}$.

The values and the distribution of the sum are represented by the product of the two polynomials. We have the following equation.
\[\sum_{i=1}^{m} x^{a_{i}} \cdot \sum_{i=1}^{m} x^{b_{i}} = \left( \sum_{i=1}^{m} x^{i} \right)^{2}\]
\[\sum_{i=1}^{m} x^{a_{i} -1} \cdot \sum_{i=1}^{m} x^{b_{i} -1} = \left( \sum_{i=1}^{m} x^{i -1} \right)^{2} = \left( \sum_{i=0}^{m-1} x^{i} \right)^{2} =\]
\[= \left( \frac{x^{m} -1}{x-1} \right) ^{2} = \left( \prod_{1 < d \mid m} \Phi_{d}(x) \right) ^{2} = \prod_{1 < d \mid m} \Phi_{d}(x)^{2}\]
Since the cyclotomic polynomials $\Phi_{d}(x)$ are irreducible over the integers, $f$ and $g$ both have to be the product of some of them.

The games of size $2$ can be identified with the ordered pairs $(f, g)$.
\begin{theorem}[Theorem 1 in \cite{GR}]
A polynomial $P(x)$ is a solution to some game with dice-size $m$ if and only if
\begin{itemize}
    \item $P(x)$ has nonnegative, integral coefficients;
    \item $P(x)$ is monic;
    \item $P(1) = m$;
    \item $P(x) /x$ is a polynomial, all of whose roots are $m$th roots of unity.
\end{itemize}
\end{theorem}
Now let $p$ be a prime. We will use the fact that for any positive integer $j$, $\Phi_{p^{j}}$ have nonnegative coefficients and $\Phi_{p^{j}}(1) = p$.
\begin{theorem}
Let $m = p^{k}$ be a prime power. Then $S(2, p^{k})$ depends only on $k$. If $k \ge 4$, then
\[\frac{3^{k} +(-1)^{k} -2}{2k} < S(2, p^{k}) < \frac{3^{k} +(-1)^{k} -2}{4}\]
\end{theorem}
$S(2, p) = 0$ as a special case of Theorem 3 in \cite{GR}.

$S(2, p^{2}) = 2$ as a special case of Theorem 2 in \cite{GR}.

$S(2, p^{3}) = 6$. This occurs for $p = 2$ in \cite{GR} when the case $m = 8$ is used to illustrate the method.
\begin{proof}
\[f(x) g(x) = \prod_{j = 1}^{k} \Phi_{p^{j}}(x)^{2}\]
On the right hand side there are $2k$ polynomial factors. $f(1) = g(1) = p^{k}$ implies that $f$ and $g$ both have to have $k$ such factors. Given that this is satisfied, the distribution of the factors between $f$ and $g$ is arbitrary.

The factor $\Phi_{p^{j}}(x)$ corresponding to the exponent $j$ can appear in the solution in three ways: two times in $f$, two times in $g$, or one time in $f$ and one time in $g$. The number of exponents appearing two times in $f$ and of those appearing two times in $g$ have to be equal. Let $l$ be this number. Hence $2l \le k$. The number of such solutions is
\[\binom{k}{l} \binom{k-l}{l} = \frac{k!}{l!^{2} (k -2l)!} = \binom{k}{2l} \binom{2l}{l}\]
Particularly, $l = 0$ implies $f = g$ and gives the standard game. In any other case, the polynomials $f$ and $g$ are distinct and uniquely determine each other.

The number of all the nonstandard solutions is
\[S(2, p^{k}) = \sum_{l=1}^{\lfloor k/2 \rfloor} \frac{k!}{l!^{2} (k -2l)!} = \sum_{l=1}^{\lfloor k/2 \rfloor} \binom{2l}{l} \binom{k}{2l}\]
\begin{lemma}
\[\sum_{l=0}^{\lfloor k/2 \rfloor} \binom{k}{2l} 2^{2l} = \frac{3^{k} +(-1)^{k}}{2}\]
\end{lemma}
\[3^{k} +(-1)^{k} \equiv 2 \cdot (-1)^{k} \equiv 2 \mod 4 \Rightarrow \frac{3^{k} +(-1)^{k}}{2} \equiv 1 \mod 2\]
\begin{proof}
\[(1 \pm 2)^{k} = \sum_{l=0}^{k} \binom{k}{l} (\pm 2)^{l} = \sum_{l=0}^{\lfloor k/2 \rfloor} \binom{k}{2l} 2^{2l} \pm \sum_{l=0}^{\lceil k/2 \rceil -1} \binom{k}{2l+1} 2^{2l+1}\]
\[2 \sum_{l=0}^{\lfloor k/2 \rfloor} \binom{k}{2l} 2^{2l} = (1 +2)^{k} +(1 -2)^{k} = 3^{k} +(-1)^{k} \Rightarrow \sum_{l=0}^{\lfloor k/2 \rfloor} \binom{k}{2l} 2^{2l} = \frac{3^{k} +(-1)^{k}}{2}\]
\end{proof}
For both the upper and the lower bound, we use the Lemma in the form
\[\sum_{l=1}^{\lfloor k/2 \rfloor} \binom{k}{2l} 2^{2l} = \frac{3^{k} +(-1)^{k}}{2} -1\]

The upper bound:
\[\binom{2l}{l} \le \sum_{i=0}^{l-1} \binom{2l}{2i+1} = \frac{1}{2} \cdot 2^{2l}\text{ }(2 \nmid l \ge 1)\]
\[\binom{2l}{l} < \sum_{i=0}^{l-1} \binom{2l}{2i+1} = \frac{1}{2} \cdot 2^{2l}\text{ }(2 \nmid l \ge 3)\]
\[\binom{2l}{l} < \sum_{i=0}^{l} \binom{2l}{2i} = \frac{1}{2} \cdot 2^{2l}\text{ }(2 \mid l \ge 2)\]
\[S(2, p^{k}) = \sum_{l=1}^{\lfloor k/2 \rfloor} \binom{k}{2l} \binom{2l}{l} \le \frac{1}{2} \sum_{l=1}^{\lfloor k/2 \rfloor} \binom{k}{2l} 2^{2l} =\]
\[= \frac{1}{2} \left( \frac{3^{k} +(-1)^{k}}{2} -1 \right) = \frac{3^{k} +(-1)^{k} -2}{4} \text{ }(k \ge 2)\]
The inequality
\[\binom{2l}{l} \le \frac{1}{2} \cdot 2^{2l}\]
is strict except if $l = 1$. For $k \ge 4$, there is a term of the sum such that $l > 1$.

The lower bound:
\[\binom{2l}{l} \ge \frac{2^{2l}}{2l}\text{ }(l \ge 1)\]
\[\frac{2^{2l}}{2l} > \frac{2^{2l}}{k}\text{ }(2l < k)\]
\[S(2, p^{k}) = \sum_{l=1}^{\lfloor k/2 \rfloor} \binom{k}{2l} \binom{2l}{l} = 
\sum_{l=1}^{\lfloor k/2 \rfloor} \binom{k}{2l} \binom{2l}{l} \ge \sum_{l=1}^{\lfloor k/2 \rfloor} \binom{k}{2l} \frac{2^{2l}}{2l} >\]
\[> \frac{1}{k} \sum_{l=1}^{\lfloor k/2 \rfloor} \binom{k}{2l} 2^{2l} = \frac{1}{k} \left( \frac{3^{k} +(-1)^{k}}{2} -1 \right) \ge\]
\[= \frac{3^{k} +(-1)^{k} -2}{2k}\text{ }(k \ge 2)\]
The inequality
\[\binom{2l}{l} \ge \frac{2^{2l}}{k}\] is strict except $2l = k$. For $k \ge 4$, there is a term of the sum such that $2l < k$.
\end{proof}
\subsection*{Acknowledgement}
I would like to acknowledge to the Algebra and Number Theory Research Seminar at the Eötvös Loránd University where I heard about this problem from Csaba Szabó and the method used from Péter Frenkel.

\end{document}